\documentclass{amsart}
\usepackage{amsmath, amssymb,epic,graphicx,mathrsfs,enumerate}
\usepackage[all]{xy}

\usepackage{amsthm}
\usepackage{amssymb}
\usepackage{latexsym}
\usepackage{longtable}
\usepackage{epsfig}
\usepackage{amsmath}
\usepackage{hhline}

\newtheorem{thm}{Theorem}[section]
\newtheorem{prop}[thm]{Proposition}
\newtheorem{lem}[thm]{Lemma}
\newtheorem{cor}[thm]{Corollary}

\newtheorem{definition}[thm]{Definition}

\begin{document}

\title{Covering certain wreath products with proper subgroups}
\author{Martino Garonzi and Attila Mar\'oti}
\thanks{The research of the second author was supported by a Marie
Curie International Reintegration Grant within the 7th European
Community Framework Programme and partially by grants OTKA T049841
and OTKA NK72523.}
\date{14th of March, 2010}
\maketitle

\begin{abstract}
For a non-cyclic finite group $X$ let $\sigma(X)$ be the least
number of proper subgroups of $X$ whose union is $X$. Precise
formulas or estimates are given for $\sigma(S \wr C_{m})$ for
certain nonabelian finite simple groups $S$ where $C_m$ is a
cyclic group of order $m$.
\end{abstract}

\section{Introduction}

For a non-cyclic finite group $X$ let $\sigma(X)$ be the least
number of proper subgroups of $X$ whose union is $X$. Let $S$ be a
nonabelian finite simple group, let $\Sigma$ be a nonempty subset
of $S$, and let $m$ be a positive integer. Let $\alpha(m)$ be the
number of distinct prime divisors of $m$. Let $\mathcal{M}$ be a
nonempty set of maximal subgroups of $S$ with the following
properties (provided that such an $\mathcal{M}$ exists).
\begin{enumerate}
\item[(0)] If $M \in \mathcal{M}$ then $M^{s} \in \mathcal{M}$ for
any $s \in S$;

\item[(1)] $\Sigma \cap M \neq \emptyset$ for every $M \in
\mathcal{M}$;

\item[(2)] $\Sigma \subseteq \bigcup_{M \in \mathcal{M}} M$;

\item[(3)] $\Sigma \cap M_{1} \cap M_{2} = \emptyset$ for every
distinct pair of subgroups $M_{1}$ and $M_{2}$ of $\mathcal{M}$;

\item[(4)] $\mathcal{M}$ contains at least two subgroups that are
not conjugate in $S$;

\item[(5)] $m \geq 2$ and
$$\max \Big\{ (1+\alpha(m)){|S|}^{m/\ell}, \underset{H <
S}{\underset{H \not\in \mathcal{M}}{\max}} |\Sigma \cap
H|{|H|}^{m-1} \Big\} \leq$$
$$\leq \min \Big\{ \Big( \sum |\Sigma \cap M_{1}| |\Sigma \cap M_{2}|
\Big) {|S|}^{m-2}, \underset{M \in \mathcal{M}}{\min} |\Sigma \cap
M|{|M|}^{m-1} \Big\}$$ where $\ell$ is the smallest prime divisor
of $m$ and the sum is over all pairs $(M_{1},M_{2}) \in
{\mathcal{M}}^{2}$ with $M_{1}$ not conjugate to $M_{2}$.
\end{enumerate}

Let $\mathcal{N}$ denote a covering for $S$, that is, a set of
proper subgroups of $S$ whose union is $S$.

\begin{thm}
\label{main1} Using the notations and assumptions introduced above
we have
$$\alpha(m) + \sum_{M \in \mathcal{M}} {|S:M|}^{m-1} \leq
\sigma(S \wr C_{m}) \leq \alpha(m) + \min_{\mathcal{N}} \sum_{M
\in \mathcal{N}} {|S:M|}^{m-1}.$$
\end{thm}

We state and prove two direct consequences of Theorem \ref{main1}.
Recall that $M_{11}$ is the Mathieu group of degree $11$.

\begin{cor}
\label{c1}
For every positive integer $m$ we have $$\sigma(M_{11}
\wr C_{m}) = \alpha(m) + 11^{m} + 12^{m}.$$
\end{cor}

Let $PSL(n,q)$ denote the projective special linear group of
dimension $n$ over a field of order $q$.

\begin{cor}
\label{c2} Let $p$ be a prime at least $11$ and $m$ be a positive
integer with smallest prime divisor at least $5$. Then
$$\sigma(PSL(2,p)
\wr C_{m}) = \alpha(m) + {(p+1)}^{m} + {(p(p-1)/2)}^{m}.$$
\end{cor}

The ideas of the proof of Theorem \ref{main1} together with the
ideas in \cite{BEGHM} can be used to find a formula for
$\sigma(PSL(n,q) \wr C_{m})$ holding for several infinite series
of groups $PSL(n,q) \wr C_{m}$ for $n \geq 12$. However, since
such an investigation would be quite lengthy, we do not pursue it
in this paper.

Let $A_n$ be the alternating group of degree $n$ where $n$ is at
least $5$. The ideas of the proof of Theorem \ref{main1} together
with the ideas in \cite{Ma2} can be used to find a formula and
some estimates for $\sigma(A_{n} \wr C_{m})$ in various cases.


\begin{thm}
\label{main2} Let us use the notations and assumptions introduced
above. Let $n$ be larger than $12$. If $n$ is congruent to $2$
modulo $4$ then
$$\sigma(A_{n} \wr C_{m}) = \alpha(m) +
\overset{(n/2)-2}{\underset{i \;
\mathrm{odd}}{\underset{i=1}{\sum}}} {{n \choose i}}^{m} +
\frac{1}{2^{m}} {{n \choose n/2}}^{m}.$$ Otherwise, if $n$ is not
congruent to $2$ modulo $4$, then
$$\alpha(m) + \frac{1}{2} \overset{n}{\underset{i \;
\mathrm{odd}}{\underset{i=1}{\sum}}} {{n \choose i}}^{m} \leq
\sigma(A_{n} \wr C_{m}).$$
\end{thm}

In some sense Theorem \ref{main2} extends a theorem of \cite{Ma2},
namely that $2^{n-2} \leq \sigma(A_{n})$ if $n > 9$ with equality
if and only if $n$ is congruent to $2$ modulo $4$.

Finally we show the following result using the ideas of Theorem
\ref{main1}.

\begin{thm}
\label{main3} Let us use the notations and assumptions introduced
above. Let $n$ be a positive integer with a prime divisor at most
$\sqrt[3]{n}$. Then $\sigma(A_{n} \wr C_m)$ is asymptotically
equal to
$$\alpha(m) + \min_{\mathcal{N}} \sum_{M \in \mathcal{N}}
{|A_{n}:M|}^{m-1}$$ as $n$ goes to infinity.
\end{thm}

Theorem \ref{main1} and Corollaries \ref{c1}, \ref{c2} are
independent from the Classification of Finite Simple Groups
(CFSG). Theorems \ref{main2} and \ref{main3} do depend on CFSG,
but with more work using \cite{PS} instead of \cite{Ma1} one can
omit CFSG from the proofs.


There are many papers on the topic of covering groups with proper
subgroups. The first of these works \cite{S} appeared in 1926. The
systematic study of the invariant $\sigma(X)$ was initiated in
\cite{C}. Since then a lot of papers appeared in this subject
including \cite{T}, \cite{DL}, and \cite{G}.

A finite group $X$ is called $\sigma$-elementary (or
$\sigma$-primitive) if for any proper, nontrivial normal subgroup
$N$ of $X$ we have $\sigma(X) < \sigma(X/N)$. $\sigma$-elementary
groups play a crucial role in determining when $\sigma(X)$ can
equal a given positive integer $n$ for some finite group $X$. The
groups we consider in this paper are $\sigma$-elementary. Giving
good lower bounds for $\sigma(X)$ for $\sigma$-elementary groups
$X$ will help answer the problem of what the density of those
positive integers $n$ is for which there exists a finite group $G$
with $n = \sigma(G)$.

\section{On subgroups of product type}

Let $S$ be a nonabelian finite simple group, and let $G = S \wr
C_m$ be the wreath product of $S$ with the cyclic group $C_{m}$ of
order $m$. Denote by $\gamma$ a generator of $C_m$. If $M$ is a
maximal subgroup of $S$ and $g_1, \ldots ,g_m$ are elements of
$S$, the normalizer in $G$ of
$$M^{g_1} \times \cdots \times M^{g_m} \leq S^m = \mathrm{soc}(G)$$
is called a subgroup of product type. A subgroup of product type
is maximal in $G$ (but we will not use this fact in the paper). In
the following let the subscripts of the $g_i$'s and the $x_i$'s be
modulo $m$.

\begin{lem}
\label{lb} Let $M$ be a maximal subgroup of $S$, and let $k \in
\{1, \ldots ,m-1\}$. Let $g_{1}, \ldots ,g_m$ be elements of $S$
with $g_{1}=1$. Choose $\gamma := (1, 2, \ldots , m)$. The element
$(x_1, \ldots ,x_m)\gamma^k$ belongs to $N_G(M \times M^{g_2}
\times \cdots \times M^{g_m})$ if and only if
$$x_{i-k} \in g_{i-k}^{-1}Mg_i \hspace{2cm} \forall i=1, \ldots ,m.$$
In particular, if $t$ is any positive integer at most $m$ and
$(x_1, \ldots ,x_m)\gamma^k$ belongs to $N_G(M \times M^{g_2}
\times \cdots \times M^{g_m})$, then
$$x_{t}x_{k+t}x_{2k+t} \cdots x_{(l-1)k+t} \in M^{g_{t}},$$ where $l=m/(m,k)$.
\end{lem}

\begin{proof}
The element $(x_1, \ldots ,x_m) \gamma^k$ normalizes $M^{g_1}
\times M^{g_2} \times \cdots \times M^{g_m}$ if and only if
$$(M^{g_1x_1} \times M^{g_2x_2} \times \ldots \times
M^{g_mx_m})^{\gamma^k}=M^{g_1} \times M^{g_2} \times \cdots \times
M^{g_m}.$$ The permutation $\gamma^k$ sends $i$ to $i+k$ modulo
$m$, so the condition becomes the following
$$M^{g_{1-k}x_{1-k}} \times M^{g_{2-k}x_{2-k}} \times \cdots \times
M^{g_{m-k}x_{m-k}}=M^{g_1} \times M^{g_2} \times \ldots \times
M^{g_m}.$$ That is, $$g_{i-k} x_{i-k} g_{i}^{-1} \in M
\hspace{1cm} \forall i=1, \ldots ,m.$$ Multiplying on the right by
$g_i$ and on the left by $g_{i-k}^{-1}$ we obtain $$x_{i-k} \in
g_{i-k}^{-1} M g_i \hspace{1cm} \forall i=1, \ldots ,m.$$ Let $t$
be a positive integer at most $m$. The line with $x_t$ on the
left-hand side says that $x_t \in g_{t}^{-1}M g_{k+t}$; the line
with $x_{k+t}$ on the left-hand side says that $x_{k+t} \in
g_{k+t}^{-1}Mg_{2k+t}$, and so on. By multiplying these together
in this order we obtain that $x_{t}x_{t+k}x_{t+2k} \cdots
x_{t+(l-1)k} \in M^{g_{t}}$, where $l$ is the smallest number at
most $m$ such that $m$ divides $lk$, that is, $l=m/(m,k)$.
\end{proof}

\section{An upper bound for $\sigma(S \wr C_{m})$}

\begin{prop}
\label{bdsn} Let $S$ be a nonabelian finite simple group, let
$\mathcal{N}$ denote a covering for $S$, let $m$ be a fixed
positive integer, and let $\alpha(m)$ denote the number of
distinct prime factors of $m$. Then
$$\sigma(S \wr C_m) \leq \alpha(m) + \min_{\mathcal{N}}
\sum_{M \in \mathcal{N}} {|S:M|}^{m-1}.$$
\end{prop}

\begin{proof}
The bound is clearly true for $m = 1$. Assume that $m > 1$.

The idea is to construct a covering of $S \wr C_m$ which consists
of exactly
$$\alpha(m) + \min_{\mathcal{N}} \sum_{M \in \mathcal{N}} {|S:M|}^{m-1}$$
proper subgroups.

There are $\alpha(m)$ maximal subgroups of the group $S \wr C_{m}$
containing its socle. Choose all of these to be in the covering.
Then we are left to cover all elements of the form $(x_{1}, \ldots
, x_{m})\gamma^{k}$ where the $x_{i}$'s are elements of $S$, where
$C_{m} = \langle \gamma \rangle$, and $k$ is coprime to $m$. It
suffices to show that such elements can be covered by the
subgroups of the form $$N_G(M \times M^{g_2} \times M^{g_3} \times
\cdots \times M^{g_m})$$ where $M$ varies in a fixed cover
$\mathcal{N}$ of $S$ and the $g_i$'s vary in $S$, because for each
fixed $M$ in $\mathcal{N}$ we have $|S:M|$ choices for $M^{g_i}$
for each $i \in \{2, \ldots ,m\}$.

By Lemma \ref{lb}, $(x_1, \ldots ,x_m)\gamma^k$ belongs to $N_G(M
\times M^{g_2} \times \cdots \times M^{g_m})$ if and only if
$$x_{i-k} \in g_{i-k}^{-1}Mg_i \hspace{2cm} \forall
i=1, \ldots ,m,$$ with $g_1=1$. The first condition is $x_{1-k}
\in g_{1-k}^{-1}M$. Choose $g_{1-k}=x_{1-k}^{-1}$. Then move to
the condition $x_{j-k} \in g_{j-k}^{-1}Mg_j$ with $j=1-k$, i.e.
$x_{1-2k} \in g_{1-2k}^{-1}Mg_{1-k}$, and rewrite it using the
information $g_{1-k}=x_{1-k}^{-1}$: get $x_{1-2k}x_{1-k} \in
g_{1-2k}^{-1}M$. Choose $g_{1-2k}=x_{1-k}^{-1}x_{1-2k}^{-1}$.
Continue this process for $m/(m,k)=m$ iterations, using Lemma
\ref{lb} (recall that $m$ is coprime to $k$). Choose
$$g_{1-jk}=x_{1-k}^{-1}x_{1-2k}^{-1} \cdots x_{1-jk}^{-1},
\hspace{1cm} \forall j=1, \ldots ,m-1.$$ At the $m$-th time we get
the relation $$x_{1-mk}x_{1-(m-1)k} \cdots x_{1-2k}x_{1-k} \in
g_{1-mk}^{-1}M.$$ But $g_{1-mk}=g_1 \in M$, so to conclude it
suffices to choose an $M$ from $\mathcal{N}$ which contains the
element $x_{1-mk}x_{1-(m-1)k} \cdots x_{1-2k}x_{1-k}$.
\end{proof}

\section{On subgroups of diagonal type}

Let $S$ be a nonabelian finite simple group. Let $m$ be a positive
integer at least $2$ and let $t$ be a divisor of $m$ which is less
than $m$. For positive integers $i$ and $j$ with $1 \leq i \leq t$
and $2 \leq j \leq m/t$ let $\varphi_{i,j}$ be an automorphism of
$S$. For simplicity, let us denote the matrix
$(\varphi_{i,j})_{i,j}$ by $\varphi$. Let $$\Delta_{\varphi} = \{
(y_{1}, \ldots , y_{t}, y_{1}^{\varphi_{1,2}}, \ldots ,
y_{t}^{\varphi_{t,2}}, \ldots , y_{1}^{\varphi_{1,m/t}}, \ldots,
y_{t}^{\varphi_{t,m/t}}) | y_{1}, \ldots, y_{t} \in S \}$$ which
is a subgroup of $S^{m} = \mathrm{soc}(G)$ where $G = S \wr
C_{m}$. The subgroup $N_{G}(\Delta_{\varphi})$ is called a
subgroup of diagonal type.

Consider the restriction to $N_{G}(\Delta_{\varphi})$ of the
natural projection of $G$ onto $C_m$. Any element of $C_m$ has
preimage of size at most $|\Delta_{\varphi}| \leq {|S|}^{m/\ell}$
where $\ell$ is the smallest prime divisor of $m$.

\section{Definite unbeatability}

The following definition was introduced in \cite{Ma2}.

\begin{definition}
\label{d1} Let $X$ be a finite group. Let $\mathcal{H}$ be a set
of proper subgroups of $X$, and let $\Pi \subseteq X$. Suppose
that the following four conditions hold on $\mathcal{H}$ and
$\Pi$.
\begin{enumerate}
\item $\Pi \cap H \neq \emptyset$ for every $H \in \mathcal{H}$;

\item $\Pi \subseteq \bigcup_{H \in \mathcal{H}} H$;

\item $\Pi \cap H_{1} \cap H_{2} = \emptyset$ for every distinct
pair of subgroups $H_{1}$ and $H_{2}$ of $\mathcal{H}$;

\item $|\Pi \cap K| \leq |\Pi \cap H|$ for every $H \in
\mathcal{H}$ and $K < X$ with $K \not \in \mathcal{H}$.
\end{enumerate}
Then $\mathcal{H}$ is said to be definitely unbeatable on $\Pi$.
\end{definition}

For $\Pi \subseteq X$ let $\sigma(\Pi)$ be the least cardinality
of a family of proper subgroups of $X$ whose union contains $\Pi$.
The next lemma is straightforward so we state it without proof.

\begin{lem}
\label{l6} If $\mathcal{H}$ is definitely unbeatable on $\Pi$ then
$\sigma(\Pi)=|\mathcal{H}|$.
\end{lem}

It follows that if $\mathcal{H}$ is definitely unbeatable on $\Pi$
then $|\mathcal{H}| = \sigma(\Pi) \leq \sigma(X)$.

\section{Proof of Theorem \ref{main1}}

In this section we prove Theorem \ref{main1}.

By Proposition \ref{bdsn}, it is sufficient to show the lower
bound of the statement of Theorem \ref{main1}.

Fix a positive integer $m$ at least $2$, let $S$ be a nonabelian
finite simple group, and let $\Sigma$ and $\mathcal{M}$ be as in
the Introduction (satisfying conditions (0)-(5)). As before, let
$G = S \wr C_{m}$.


Let $\Pi_{1}$ be the set consisting of all elements $(x_{1},
\ldots , x_{m})\gamma$ of $G$ with the property that $x_{1}\cdots
x_{m} \in \Sigma$ and let $\mathcal{H}_{1}$ be the set consisting
of all subgroups $N_{G}(M \times M^{g_{2}} \times \cdots \times
M^{g_{m}})$ with the property that $M \in \mathcal{M}$. For fixed
$M \in \mathcal{M}$ put $$\Sigma_{M} = \Sigma \cap \Big(
\bigcup_{s \in S} M^{s} \Big).$$ Note that, by Conditions (0) and
(3) of the Introduction, $\Sigma_{M} \cap \Sigma_{K} = \emptyset$
if $M$ and $K$ are non-conjugate elements of $\mathcal{M}$. Let
$\Pi_{2}$ be the set consisting of all elements $(x_{1}, \ldots ,
x_{m})\gamma^{r}$ of $G$ with the property that $r$ is a prime
divisor of $m$ and that $x_{1}x_{r+1}\cdots x_{m-r+1}$ is in
$\Sigma_{M}$ and $x_{2}x_{r+2} \cdots x_{m-r+2}$ is in
$\Sigma_{K}$ where $M$ and $K$ are not conjugate in $S$. Finally,
let $\mathcal{H}_{2}$ be the set consisting of all maximal
subgroups of $G$ containing the socle of $G$. Put $\Pi = \Pi_{1}
\cup \Pi_{2}$ and $\mathcal{H} = \mathcal{H}_{1} \cup
\mathcal{H}_{2}$. By Lemma \ref{l6} and the remark following Lemma
\ref{l6}, the following proposition finishes the proof of Theorem
\ref{main1}.

\begin{prop}
\label{p1} The set $\mathcal{H}$ of subgroups of $G$ is definitely
unbeatable on $\Pi$.
\end{prop}

\begin{proof}
In this paragraph let us prove Condition (1) of Definition
\ref{d1}. Let $H$ be an arbitrary subgroup in $\mathcal{H}_{1}$.
Suppose that $H = N_{G}(M \times M^{g_{2}} \times \cdots \times
M^{g_{m}})$ for some $M \in \mathcal{M}$ and $g_{2}, \ldots ,
g_{m} \in S$. Let $\pi$ be an element of $\Sigma \cap M$. (Such an
element exists by Condition (1) of the Introduction.) Let $x_{1} =
g_{2}$, $x_{2} = g_{2}^{-1}g_{3}, \ldots , x_{m-1} =
g_{m-1}^{-1}g_{m}$, and $x_{m} = x_{m-1}^{-1} \cdots x_{2}^{-1}
x_{1}^{-1} \pi$. Then, by Lemma \ref{lb}, the element $(x_{1},
\ldots, x_{m}) \gamma$ is in $H$ (and also in $\Pi_{1}$). Let $H$
be an arbitrary subgroup in $\mathcal{H}_{2}$. Let the index of
$H$ in $G$ be $r$ for some prime divisor $r$ of $m$. Then $H$
contains every element of $\Pi_{2}$ of the form $(x_{1}, \ldots ,
x_{m})\gamma^{r}$.

In this paragraph let us prove Condition (2) of Definition
\ref{d1}. Let $(x_{1}, \ldots , x_{m}) \gamma$ be an arbitrary
element of $\Pi_{1}$. We will show that there exists an $H \in
\mathcal{H}_{1}$ which contains $(x_{1}, \ldots , x_{m}) \gamma$.
We know that $x_{1}x_{2}\cdots x_{m} \in \Sigma$. By Condition (2)
of the Introduction, we see that there exists an $M \in
\mathcal{M}$ with the property that $x_{1}x_{2} \cdots x_{m} \in
M$. Now let $g_{2} = x_{1}$, $g_{3} = x_{1}x_{2}, \ldots , g_{m} =
x_{1}x_{2}\cdots x_{m-1}$. Then $H = N_{G}(M \times M^{g_{2}}
\times \cdots \times M^{g_{m}})$ contains $(x_{1}, \ldots , x_{m})
\gamma$ by Lemma \ref{lb}. Now let $(x_{1}, \ldots ,
x_{m})\gamma^{r}$ be an arbitrary element of $\Pi_{2}$. This is
contained in the maximal subgroup $H$ of index $r$ in $G$
containing the socle of $G$. We see that $H$ is contained in
$\mathcal{H}_{2}$.

Now we show that Condition (3) of Definition \ref{d1} is
satisfied. Notice that, by construction (by the second half of
Lemma \ref{lb} and by Condition (4) of the Introduction), $\Pi_{1}
\cap H_{2} = \emptyset$ and $\Pi_{2} \cap H_{1} = \emptyset$ for
every $H_{1} \in \mathcal{H}_{1}$ and $H_{2} \in \mathcal{H}_{2}$.
Hence it is sufficient to show that $\Pi_{1} \cap H_{1} \cap H_{2}
= \emptyset$ for distinct subgroups $H_{1}$ and $H_{2}$ in
$\mathcal{H}_{1}$ and also that $\Pi_{2} \cap H_{1} \cap H_{2} =
\emptyset$ for distinct subgroups $H_{1}$ and $H_{2}$ in
$\mathcal{H}_{2}$. The latter claim is clear by considering the
projection map from $G$ to $C_m$, hence it is sufficient to show
the former claim. First notice that if $M$ and $K$ are two
distinct elements of $\mathcal{M}$ and $g_{2}, \ldots , g_{m}$,
$k_{2}, \ldots , k_{m}$ are arbitrary elements of $S$, then
$$\Pi_{1} \cap N_{G}(M \times M^{g_{2}} \times \cdots \times
M^{g_{m}}) \cap N_{G}(K \times K^{k_{2}} \times \cdots \times
K^{k_{m}}) = \emptyset ,$$ by Lemma \ref{lb} and by Condition (3)
of the Introduction. Finally let $M$ be fixed and let
$$\Pi_{1} \cap N_{G}(M \times M^{g_{2}} \times \cdots \times
M^{g_{m}}) \cap N_{G}(M \times M^{k_{2}} \times \cdots \times
M^{k_{m}}) \not= \emptyset$$ for some elements $g_{2}, \ldots,
g_{m}, k_{2}, \ldots , k_{m}$ of $S$. Then by Lemma \ref{lb}, for
every index $i$ with $2 \leq i \leq m$, we have $Mg_{i}=Mk_{i}$
(just consider the products $x_{1} \cdots x_{j}$ for all positive
integers $j$ with $1 \leq j \leq m-1$ where $(x_{1}, \ldots ,
x_{m}) \gamma$ is in the intersection of $\Pi_{1}$ with the two
normalizers) from which it follows that $M^{g_{i}k_{i}^{-1}} = M$.
This finishes the proof of Condition (3) of Definition \ref{d1}.

To show that Condition (4) of Definition \ref{d1} is satisfied, it
is necessary to make three easy observations based on the
following folklore lemma.

\begin{lem}
\label{folklore} A maximal subgroup of $G = S \wr C_{m}$ either
contains the socle of $G$, is of product type, or is of diagonal
type.
\end{lem}

If $L$ is a maximal subgroup of $G$ containing the socle of $G$
then
$$|\Pi \cap L| = \Big( \sum |\Sigma \cap M_{1}| |\Sigma \cap
M_{2}| \Big) {|S|}^{m-2}$$ where the sum is over all pairs
$(M_{1}, M_{2}) \in {\mathcal{M}}^{2}$ such that $M_{1}$ is not
conjugate to $M_{2}$ in $S$. If $L$ is of product type, then $|\Pi
\cap L| = |\Sigma \cap M|{|M|}^{m-1}$ where $M$ is such that $L =
N_{G}(M \times M^{g_2} \times \cdots \times M^{g_m})$ for some
elements $g_{2}, \ldots , g_{m}$ of $S$. Finally if $L$ is of
diagonal type, then $|\Pi \cap L| \leq (1+ \alpha(m))
{|S|}^{m/\ell}$ where $\ell$ is the smallest prime divisor of $m$.
Putting these observations together, Condition (5) of the
Introduction gives Condition (4) of Definition \ref{d1}.
\end{proof}

\section{Proof of Corollary \ref{c1}}

Corollary \ref{c1} is clear for $m = 1$ by \cite{Ma2}, so let us
assume that $m \geq 2$.

Let $\mathcal{M}$ be the set of all $11$ conjugates of the maximal
subgroup $M_{10}$ of $M_{11}$ together with all $12$ conjugates of
the maximal subgroup $PSL(2,11)$ of $M_{11}$. It is easy to check
that $\mathcal{M}$ is a covering for $M_{11}$, hence, by the upper
bound of Theorem \ref{main1}, we have $\sigma(M_{11} \wr C_{m})
\leq \alpha(m) + 11^{m} + 12^{m}$.

Let $\Sigma$ be the subset of $M_{11}$ consisting of all elements
of orders $8$ or $11$. To prove Corollary \ref{c1} it is
sufficient to show that $\Sigma$ and $\mathcal{M}$ satisfy the six
conditions of the statement of Theorem \ref{main1}.

By \cite{GAP} we know that the maximal subgroups of $M_{11}$ are:
$M_{10}$, $PSL(2,11)$, $M_{9}:2$, $S_{5}$, and $M_{8}:S_{3}$, and
that for these we have the following.

\begin{itemize}
\item $M_{10}$ has order $720$, it contains $180$ elements of
order $8$ and no element of order $11$; no element of order $8$ is
contained in two distinct conjugates of $M_{10}$; \item
$PSL(2,11)$ has order $660$, it contains no element of order $8$
and $120$ elements of order $11$; no element of order $11$ is
contained in two distinct conjugates of $PSL(2,11)$; \item $M_9:2$
has order $144$, it contains $36$ elements of order $8$ and no
element of order $11$; \item $S_5$ has order $120$, it contains no
element of order $8$ and no element of order $11$; \item
$M_{8}:S_{3}$ has order $48$, it contains $12$ elements of order
$8$ and no element of order $11$.
\end{itemize}

This shows that the first five conditions of the statement of
Theorem \ref{main1} are verified. Now let us compute the four
expressions involved in Condition (5).

\begin{itemize}
\item $(1+\alpha(m))|S|^{m/\ell} \leq (1+\alpha(m))|S|^{m/2} =
(1+\alpha(m)) (\sqrt{7920})^m$; \item $\max_{H \not \in
\mathcal{M},\ H<S} |\Sigma \cap H||H|^{m-1} = 36 \cdot 144^{m-1}$;
\item $(\sum |\Sigma \cap M_1| |\Sigma \cap M_2|)|S|^{m-2} = 2
\cdot 132 \cdot 180 \cdot 120 \cdot 7920^{m-2}$ since we have $2
\cdot 12 \cdot 11 = 2 \cdot 132$ choices for the pair $(M_1,M_2)$;
\item $\min_{M \in \mathcal{M}}|\Sigma \cap M| |M|^{m-1} = 120
\cdot 660^{m-1}$.
\end{itemize}

We have then to prove that $$\max((1+\alpha(m))7920^{m/2},36 \cdot
144^{m-1}) \leq$$ $$\leq \min(2 \cdot 132 \cdot 180 \cdot 120
\cdot 7920^{m-2},120 \cdot 660^{m-1}).$$ Clearly the right-hand
side is $120 \cdot 660^{m-1}$ and it is bigger than $36 \cdot
144^{m-1}$, so we have to prove that $$(1+\alpha(m))7920^{m/2}
\leq 120 \cdot 660^{m-1}.$$ After rearranging, taking roots, and
using the fact that $(1+\alpha(m))^{1/m} \leq \sqrt{2}$ we obtain
that it suffices to prove the inequality $$\sqrt{2}
\frac{\sqrt{7920}}{660} \leq \left(\frac{120}{660}
\right)^{1/m}.$$ Since the right-hand side of the previous
inequality is increasing with $m$, it suffices to assume that
$m=2$. But then the inequality becomes clear.

\section{Proof of Corollary \ref{c2}}

Note that Corollary \ref{c2} is clear for $m = 1$ by \cite{BFS}.

Let $p \geq 11$ be a prime and assume that the smallest prime
divisor $\ell$ of $m$ is at least $5$.

Let $\mathcal{M}$ be the set of all $p+1$ conjugates of the
maximal subgroup $C_p \rtimes C_{(p-1)/2}$ of $PSL(2,p)$ together
with all $p(p-1)/2$ conjugates of the maximal subgroup $D_{p+1}$
of $PSL(2,p)$. It is easy to check that $\mathcal{M}$ is a
covering for $PSL(2,p)$, hence, by the upper bound of Theorem
\ref{main1}, we have $$\sigma(PSL(2,p) \wr C_{m}) \leq \alpha(m) +
{(p+1)}^{m} + {(p(p-1)/2)}^{m}.$$

Let $\Sigma_{1} \subseteq PSL(2,p)$ be a set of $p^{2}-1$ elements
each of order $p$ with the property that every element of
$\Sigma_{1}$ fixes a unique point on the projective line and that
$(\Sigma_{1} \cap M) \cup \{ 1 \}$ is a group of order $p$ for
every conjugate $M$ of $C_p \rtimes C_{(p-1)/2}$. Let $\Sigma_{2}$
be the set of all irreducible elements of $PSL(2,p)$ of order
$(p+1)/2$. Put $\Sigma = \Sigma_{1} \cup \Sigma_{2}$. To prove
Corollary \ref{c2} it is sufficient to show that $\Sigma$ and
$\mathcal{M}$ satisfy the six conditions of the statement of
Theorem \ref{main1}.

By \cite{Dickson} the maximal subgroups of $PSL(2,p)$ are the
following.

\begin{itemize}
\item $C_p \rtimes C_{(p-1)/2}$; \item $D_{p-1}$ if $p \geq 13$;
\item $D_{p+1}$; \item $A_5$, $A_4$, and $S_4$ for certain
infinite families of $p$.
\end{itemize}

Since $p \geq 11$, no element of $\Sigma$ is contained in a
subgroup of the form $A_5$, $A_4$, or $S_4$. Moreover since
$(p+1)/2$ and $p$ do not divide $p-1$, no element of $\Sigma$ is
contained in a subgroup of the form $D_{p-1}$. Similarly, it is
easy to see that no element of $\Sigma_{1}$ is contained in a
conjugate of $D_{p+1}$ and no element of $\Sigma_{2}$ is contained
in a conjugate of $C_p \rtimes C_{(p-1)/2}$.

By the above and by a bit more, it follows that the first five
conditions of the statement of Theorem \ref{main1} hold. Now let
us compute the four expressions involved in Condition (5).

But before we do so, let us note two things. If $M$ is a maximal
subgroup of the form $D_{p+1}$, then $|\Sigma \cap
M|=\varphi((p+1)/2)$ where $\varphi$ is Euler's function.
Moreover, if $M$ is conjugate to $C_p \rtimes C_{(p-1)/2}$, then
$|\Sigma \cap M|=p-1$.

\begin{itemize}
\item $$(1+\alpha(m))|S|^{m/\ell} \leq
(1+\alpha(m))((1/2)p(p^2-1))^{m/5};$$ \item $$\max_{H \not \in
\mathcal{M}} |\Sigma \cap H| |H|^{m-1} = 0;$$ \item $$(\sum
|\Sigma \cap M_1| |\Sigma \cap M_2|)|S|^{m-2} =$$ $$=
2(p+1)(p(p-1)/2) \varphi((p+1)/2) (p-1) ((1/2)p(p^2-1))^{m-2};$$
\item $$\min_{M \in \mathcal{M}} |\Sigma \cap M| |M|^{m-1} =$$
$$=\min(\varphi((p+1)/2)(p+1)^{m-1},(p-1)(p(p-1)/2)^{m-1}) =$$
$$=\varphi((p+1)/2)(p+1)^{m-1}.$$
\end{itemize}

We are easily reduced to prove the following inequality
$$(1+\alpha(m))(p(p^2-1)/2)^{m/5} \leq \varphi((p+1)/2)
(p+1)^{m-1}.$$Using the fact that $(1+\alpha(m))^{1/m} \leq
\sqrt{2}$ we obtain that it suffices to show that $$\sqrt{2}
\frac{(p(p^2-1)/2)^{1/5}}{p+1} \leq
\left(\frac{\varphi((p+1)/2)}{p+1}\right)^{1/m}.$$ Since the
right-hand side is increasing with $m$, it suffices to assume that
$m=5$. By taking $5$-th powers of both sides we obtain $$2
\sqrt{2} p (p^2-1) \leq (p+1)^4.$$ But this is clearly true for $p
\geq 11$.

\section{Alternating groups}

>From this section on we will deal with the special case when $S$
is the alternating group $A_n$. We will repeat some of the
definitions in more elaborate form.

For each positive integer $n \geq 5$ which is not a prime we
define a subset $\Pi_{0}$ of $A_n$ and a set $\mathcal{H}_{0}$ of
maximal subgroups of $A_n$. (These sets $\Pi_{0}$ and
$\mathcal{H}_{0}$ will be close to the sets $\Sigma$ and
$\mathcal{M}$ of the Introduction.)

Let $n$ be odd (and not a prime). In this case let $\Pi_{0}$ be
the set of all $n$-cycles of $A_n$ and let $\mathcal{H}_{0}$ be
the set of all maximal subgroups of $A_n$ conjugate to $(S_{n/p}
\wr S_{p}) \cap A_n$ where $p$ is the smallest prime divisor of
$n$.

Let $n$ be divisible by $4$. In this case let $\Pi_{0}$ be the set
of all $(i,n-i)$-cycles of $A_{n}$ (permutations of $A_n$ which
are products of two disjoint cycles one of length $i$ and one of
length $n-i$) for all odd $i$ with $i < n/2$ and let
$\mathcal{H}_{0}$ be the set of all maximal subgroups of $A_n$
conjugate to some group of the form $(S_{i} \times S_{n-i}) \cap
A_{n}$ for some odd $i$ with $i < n/2$.

Let $n$ be congruent to $2$ modulo $4$. In this case let $\Pi_{0}$
be the set of all $(i,n-i)$-cycles of $A_{n}$ for all odd $i$ with
$i \leq n/2$ and let $\mathcal{H}_{0}$ be the set of all maximal
subgroups of $A_n$ conjugate to some group of the form $(S_{i}
\times S_{n-i}) \cap A_{n}$ for some odd $i$ with $i < n/2$ or
conjugate to $(S_{n/2} \wr S_{2}) \cap A_{n}$.

\begin{thm}[Mar\'oti, \cite{Ma2}]
With the notations above $\mathcal{H}_{0}$ is definitely
unbeatable on $\Pi_{0}$ provided that $n \geq 16$.
\end{thm}

\section{Wreath products}

Let $m$ be a fixed positive integer (which can be $1$). Let $G =
A_{n} \wr C_{m}$ and let $\gamma$ be a generator of $C_m$. Let
$\Pi_{1}$ be the set consisting of all elements $(x_{1}, \ldots ,
x_{m})\gamma$ of $G$ with the property that $x_{1}\cdots x_{m} \in
\Pi_{0}$ and let $\mathcal{H}_{1}$ be the set consisting of all
subgroups $N_{G}(M \times M^{g_{2}} \times \cdots \times
M^{g_{m}})$ with the property that $M \in \mathcal{H}_{0}$. If $m
= 1$, then set $\Pi = \Pi_{1}$ and $\mathcal{H} =
\mathcal{H}_{1}$. From now on, only in the rest of this paragraph,
suppose that $m
> 1$. For $n$ odd let $\Pi_{2}$ be the set consisting of all
elements $(x_{1}, \ldots , x_{m})\gamma^{r}$ of $G$ with the
property that $r$ is a prime divisor of $m$ and that
$x_{1}x_{r+1}\cdots x_{m-r+1}$ is an $n$-cycle and $x_{2}x_{r+2}
\cdots x_{m-r+2}$ is an $(n-2)$-cycle. For fixed $M \in
\mathcal{H}_{0}$ put $$\Pi_{0,M} = \Pi_{0} \cap \Big( \bigcup_{g
\in A_{n}} M^{g} \Big).$$ (Depending on $M$ (and on the parity of
$n$) $\Pi_{0,M}$ is the set of $n$-cycles or the set of
$(i,n-i)$-cycles with $i \leq n/2$ contained in the union of all
conjugates of some $M$ in $\mathcal{H}_{0}$.) For $n$ even let
$\Pi_{2}$ be the set consisting of all elements $(x_{1}, \ldots ,
x_{m})\gamma^{r}$ of $G$ with the property that $r$ is a prime
divisor of $m$ and that $x_{1}x_{r+1}\cdots x_{m-r+1} \in
\Pi_{0,M}$ and $x_{2}x_{r+2} \cdots x_{m-r+2} \in \Pi_{0,K}$ where
$M$ and $K$ are not conjugate in $A_n$. Finally, let
$\mathcal{H}_{2}$ be the set consisting of all maximal subgroups
of $G$ containing the socle of $G$. Put $\Pi = \Pi_{1} \cup
\Pi_{2}$ and $\mathcal{H} = \mathcal{H}_{1} \cup \mathcal{H}_{2}$.

\begin{prop}
\label{p2} If $m=1$, then $\mathcal{H}$ is definitely unbeatable
on $\Pi$ for $n \geq 16$. If $m > 1$, then $\mathcal{H}$ is
definitely unbeatable on $\Pi$ for $n > 12$ provided that $n$ has
a prime divisor at most $\sqrt[3]{n}$.
\end{prop}

For $m =1$ there is nothing to show. Suppose that $m > 1$.

Along the lines of the ideas in Section 6, it is possible (and
easy) to show that $\Pi$ and $\mathcal{H}$ satisfy Conditions (1),
(2), and (3) of Definition \ref{d1}. (Condition (3) of Definition
\ref{d1} is satisfied since, for example for $n$ odd, no conjugate
of $(S_{n/p} \wr S_{p}) \cap A_{n}$ contains an $(n-2)$-cycle
where $p$ is the smallest prime divisor of $n$.) Hence, to prove
Proposition \ref{p2}, it is sufficient to verify Condition (4) of
Definition \ref{d1}. This will be done in the next three sections.

\section{Some preliminary estimates}

Some of the following lemma depends on the fact that $a! (n-a)!
\geq b!(n-b)!$ whenever $a$ and $b$ are integers with $a \leq b
\leq n/2$.

\begin{lem}
Let $n$ be odd (and not a prime). Then $$|\Pi \cap H_{1}| =
|\Pi_{1} \cap H_{1}| = (1/(2^{m-1}n)) {\Big( {(n/p)!}^{p}p!
\Big)}^{m}$$ for $H_{1} \in \mathcal{H}_{1}$ where $p$ is the
smallest prime divisor of $n$, and $$|\Pi \cap H_{2}| = |\Pi_{2}
\cap H_{2}| = (2/(n(n-2))){|A_{n}|}^{m}$$ for $H_{2} \in
\mathcal{H}_{2}$. Let $n$ be divisible by $4$. Then $$|\Pi \cap
H_{1}| = |\Pi_{1} \cap H_{1}| \geq (((n/2)-2)!)((n/2)!) {\Big(
\frac{(((n/2)-1)!)(((n/2)+1)!)}{2} \Big)}^{m-1}$$ for $H_{1} \in
\mathcal{H}_{1}$. Let $n$ be congruent to $2$ modulo $4$. Then
$$|\Pi \cap H_{1}| = |\Pi_{1} \cap H_{1}| \geq
(1/2^{m-1}){(((n/2)-1)!)}^{2} {((n/2)!)}^{2m-2}$$ for $H_{1} \in
\mathcal{H}_{1}$. Finally, let $n$ be even. Then $$|\Pi \cap
H_{2}| = |\Pi_{2} \cap H_{2}| \geq \frac{4}{3(n-1)(n-3)}
{|A_{n}|}^{m}$$ for $H_{2} \in \mathcal{H}_{2}$.
\end{lem}

\begin{proof}
This follows from the above and from the observations made when
dealing with Condition (4) of Definition \ref{d1} while proving
Theorem \ref{main1}. The last statement follows from counting
$(1,n-1)$-cycles and $(3,n-3)$-cycles (twice).
\end{proof}

\begin{lem} \label{l8} Depending on $n \geq 5$ we have the following.
\begin{enumerate}

\item If $n$ is odd (and not a prime), then $$(1/(2^{m-1}n))
{\Big( {(n/p)!}^{p}p! \Big)}^{m} \leq (2/(n(n-2))){|A_{n}|}^{m},$$
hence $\min_{H \in \mathcal{H}} |\Pi \cap H| = (1/(2^{m-1}n))
{\Big( {(n/p)!}^{p}p! \Big)}^{m}$.

\item If $n$ is divisible by $4$, then $$(((n/2)-2)!)((n/2)!)
{\Big( \frac{(((n/2)-1)!)(((n/2)+1)!)}{2} \Big)}^{m-1} \leq
\frac{4}{3(n-1)(n-3)} {|A_{n}|}^{m},$$ hence $$\min_{H \in
\mathcal{H}} |\Pi \cap H| = (((n/2)-2)!)((n/2)!) {\Big(
\frac{(((n/2)-1)!)(((n/2)+1)!)}{2} \Big)}^{m-1}.$$

\item If $n$ is congruent to $2$ modulo $4$, then
$$(1/2^{m-1}){(((n/2)-1)!)}^{2} {((n/2)!)}^{2m-2}
\leq \frac{4}{3(n-1)(n-3)} {|A_{n}|}^{m},$$ hence
$$\min_{H \in \mathcal{H}} |\Pi \cap H| \geq
(1/2^{m-1}){(((n/2)-1)!)}^{2} {((n/2)!)}^{2m-2}.$$
\end{enumerate}
\end{lem}

\begin{proof}
(1) After rearranging, the inequality becomes $n-2 \leq {|S_{n}:
(S_{n/p} \wr S_p)|}^{m}$ which is clearly true.

(2) After rearranging, the inequality becomes
$$\frac{6(n-1)(n-3)}{(n+2)(n-2)} < 6 \leq {\Big(
\frac{|S_{n}|}{|S_{(n/2)-1} \times S_{(n/2)+1}|} \Big)}^{m}$$
which is clearly true.

(3) After rearranging, the inequality becomes
$$\frac{6(n-1)(n-3)}{n^{2}} < 6 < {{n \choose n/2}}^{m}$$ which is clearly true.
\end{proof}

\section{The case when $K$ is a subgroup of diagonal type}

Let $K$ be a subgroup of $G$ of diagonal type. Note that $K \not
\in \mathcal{H}$. We would like to show that $|\Pi \cap K| \leq
|\Pi \cap H|$ for every $H \in \mathcal{H}$. We have $|\Pi \cap K|
\leq (1+ \alpha(m)){|A_{n}|}^{m/2}$.

We need Stirling's formula.

\begin{thm}[Stirling's formula]
\label{Stirling} For all positive integers $n$ we have
$$\sqrt{2\pi n} {(n/e)}^{n} e^{1/(12n+1)} < n! < \sqrt{2\pi n} {(n/e)}^{n} e^{1/(12n)}.$$
\end{thm}

The declared aim of proving the inequality $|\Pi \cap K| \leq |\Pi
\cap H|$ for every $H \in \mathcal{H}$ is achieved through the
next lemma. We also point out that the right-hand sides of the
inequalities of the following lemma come from Section 11.

\begin{lem} Let $m \geq 2$. The following hold.
\label{l1}
\begin{enumerate}
\item Let $n$ be odd with smallest prime divisor $p$ at most
$\sqrt[3]{n}$. Then
$$(1+ \alpha(m)){(n!/2)}^{m/2} \leq (1/(2^{m-1}n)) {\Big(
{(n/p)!}^{p}p! \Big)}^{m}.$$

\item Let $n$ be divisible by $4$ and larger than $8$. Then $$(1+
\alpha(m)){(n!/2)}^{m/2} \leq (((n/2)-2)!)((n/2)!) {\Big(
\frac{(((n/2)-1)!)(((n/2)+1)!)}{2} \Big)}^{m-1}.$$

\item Let $n$ be congruent to $2$ modulo $4$ and larger than $10$.
Then
$$(1+ \alpha(m)){(n!/2)}^{m/2} \leq (1/2^{m-1}){(((n/2)-1)!)}^{2}
{((n/2)!)}^{2m-2}.$$
\end{enumerate}
\end{lem}

\begin{proof}
(1) It is sufficient to show the inequality $${\Big(
\frac{n}{2}(1+ \alpha(m))\Big)}^{2/m} \leq
\frac{{((n/p)!)}^{2p}{p!}^{2}}{2n!}.$$ For this it is sufficient
to see that $$n {(1+\alpha(m))}^{2/m} \leq
\frac{{((n/p)!)}^{2p}}{n!}.$$ Substituting Stirling's formula
(Theorem \ref{Stirling}) on the right-hand side, we see that it is
sufficient to show that
$$n {(1+\alpha(m))}^{2/m} \leq \frac{{(2\pi(n/p))}^{p}
{(n/pe)}^{2n}}{\sqrt{2\pi n} {(n/e)}^{n} e^{1/(12n)}}.$$ Since $3
\leq p \leq \sqrt[3]{n}$ and $e^{1/(12n)} < 2$, it is sufficient
to prove $$n {(1+\alpha(m))}^{2/m} \leq \frac{{(2 \pi
n^{2/3})}^{3} {(n^{2/3}/e)}^{2n}}{2 \sqrt{2 \pi n} {(n/e)}^{n}}.$$
Since ${(1+\alpha(m))}^{2/m} \leq 2$ it is sufficient to see that
$$\frac{\sqrt{2\pi}}{2{\pi}^{3}} \leq \frac{n^{(1/3)n +
(1/2)}}{e^{n}}.$$ But this is true for $n \geq 27$.

(2) After rearranging the inequality and taking roots we get
$${(1+\alpha(m))}^{2/m} (n!/2) \leq
{\Big(\frac{8}{n^{2}-4}\Big)}^{2/m} {\Big(
\frac{((n/2)-1)!((n/2)+1)!}{2} \Big)}^{2}.$$ Since
${(1+\alpha(m))}^{2/m} \leq 2$ and $8/(n^{2}-4) \leq
{(8/(n^{2}-4))}^{2/m}$, it is sufficient to see that
$$\frac{n^{2}-4}{2} n! \leq {(((n/2)-1)!((n/2)+1)!)}^{2}.$$ Since
${n \choose (n/2)-1} \leq 2^{n-1}$, it is sufficient to prove
$$(n^{2}-4) 2^{n-2} \leq ((n/2)-1)!((n/2)+1)!.$$ But this is true
for $n \geq 12$.

(3) After rearranging the inequality and taking roots we see that
it is sufficient to show
$$4{(1+\alpha(m))}^{2/m} {(n/2)}^{4/m} (n!/2) \leq {((n/2)!)}^{4}.$$
Since ${(1+\alpha(m))}^{2/m} \leq 2$ and ${(n/2)}^{4/m} \leq
{(n/2)}^{2}$, it is sufficient to see that $$n^{2} n! \leq
{((n/2)!)}^{4}.$$ But this can be seen by induction for $n \geq
14$.
\end{proof}

\section{The case when $K$ is a subgroup of product type}

Let $K$ be a subgroup of $G$ of product type such that $K \not \in
\mathcal{H}$. We would like to show that $|\Pi \cap K| \leq |\Pi
\cap H|$ for every $H \in \mathcal{H}$.

Suppose that $K = N_{G}(M \times M^{g_{2}} \times \cdots \times
M^{g_{m}})$ where $M$ is a maximal subgroup of $A_n$. If $M$ is an
intransitive subgroup then $\Pi \cap K = \emptyset$, by
construction of $\Pi$ and $\mathcal{H}$, hence there is nothing to
show in this case.

In the next paragraph and in Lemma \ref{l7} we will make use of
the following fact taken from \cite{Ma1}.

\begin{lem}
\label{11.1} For a positive integer $n$ at least $8$ we have
$${((n/a)!)}^{a}a! \geq {((n/b)!)}^{b}b!$$ whenever $a$ and $b$
are divisors of $n$ with $a \leq b$.
\end{lem}

Let $M$ be a maximal imprimitive subgroup of $A_n$ conjugate to $(
S_{n/a} \wr S_{a} ) \cap A_n$ for some proper divisor $a$ of $n$.
Let $n$ be odd (and not a prime). Then $\Pi_{2} \cap K =
\emptyset$ since $M$ does not contain an $(n-2)$-cycle. In this
case
$$|\Pi \cap K| = |\Pi_{1} \cap K| = (1/(2^{m-1}n)){\Big(
{((n/a)!)}^{a}a! \Big)}^{m} \leq (1/(2^{m-1}n)){\Big(
{((n/p)!)}^{p}p! \Big)}^{m}$$ and we are done by part (1) of Lemma
\ref{l8}.

Now let $n$ be even. In this case $a \geq 3$.

\begin{lem}
\label{l2} Let $n$ be even and let $a$ be the smallest divisor of
$n$ larger than $2$. If $n > 10$, then $n {((n/a)!)}^{a}a! \leq 2
{((n/2)!)}^{2}$.
\end{lem}

\begin{proof}
If $n=2a$, then we must consider the inequality $2^{a} \leq
(a-1)!$. This is clearly true if $a$ satisfies $a > 5$, hence if
$n
> 10$. This means that we may assume that $3 \leq a \leq n/4$.

The lemma is true for $10 < n \leq 28$ by inspection. From now on
we assume that $n \geq 30$.

Applying Stirling's formula (see Theorem \ref{Stirling}), we see
that it is sufficient to verify the inequality
$$n {(\sqrt{2 \pi (n/a)})}^{a} {(n/ae)}^{n} e^{a^{2}/(12n)}
\sqrt{2\pi a} {(a/e)}^{a} e^{1/(12a)} \leq 2 \pi n {(n/2e)}^{n}
e^{2/(6n+1)}.$$ After rearranging factors we obtain
$$2^{n} {(2\pi(n/a))}^{a/2} e^{a^{2}/(12n)} \sqrt{2 \pi a} {(a/e)}^{a} e^{1/(12a)}
\leq a^{n} 2 \pi e^{2/(6n+1)}.$$ After taking natural logarithms
and rearranging terms we obtain
$$a \Big( \frac{\ln (2\pi)}{2} + \frac{\ln n}{2} + \frac{\ln a}{2} + \frac{a}{12n} -1 \Big)
+ \Big( \frac{\ln a}{2} + \frac{1}{12a} - \frac{\ln (2\pi)}{2} -
\frac{2}{6n+1} \Big) \leq n (\ln a - \ln 2).$$ By the assumption
$3 \leq a \leq n/4$ and by dividing both sides of the previous
inequality by $\ln n$ we see that it is sufficient to prove $$a
\Big( 1+ \frac{\ln (2 \pi)}{2 \ln n} + \frac{1}{48 \ln n} -
\frac{1}{\ln n} \Big) + \Big( \frac{1}{2} + \frac{1}{36 \ln n} -
\frac{\ln (2\pi)}{2 \ln n} - \frac{2}{(6n+1)\ln n} \Big) \leq
\frac{n}{\ln n} (\ln a - \ln 2).$$ Since $$\frac{\ln (2 \pi)}{2
\ln n} + \frac{1}{48 \ln n} - \frac{1}{\ln n} < 0$$ and
$$\frac{1}{36 \ln n} - \frac{\ln (2\pi)}{2 \ln n} -
\frac{2}{(6n+1)\ln n} < 0,$$ it is sufficient to prove
\begin{equation}
\label{e2} \frac{a+ 0.5}{\ln a - \ln 2} \leq \frac{n}{\ln n}.
\end{equation}
This is true for $a = 3$, $4$, and $5$ (provided that $n \geq
30$). Hence assume that $7 \leq a \leq n/4$.

The function $\frac{x+ 0.5}{\ln x - \ln 2}$ increases when $x >
6$, hence it is sufficient to show inequality (\ref{e2}) in case
of the substitution $a = n/4$. But that holds for $n \geq 30$. The
proof of the lemma is now complete.
\end{proof}

By Lemma \ref{lb}, we have $|\Pi \cap K| \leq (1 + \alpha(m))
{|M|}^{m}$. The left-hand sides of Lemmas \ref{l7} and \ref{11.4}
are upper bounds for $(1 + \alpha(m)) {|M|}^{m}$ in various cases.

\begin{lem}
\label{l7} Let $n$ be even and let $a$ be the smallest divisor of
$n$ larger than $2$. Let $m \geq 2$. Then for $n > 10$ we have the
following.
\begin{enumerate}
\item If $n$ is divisible by $4$, then $$(1 + \alpha(m) ){\Big(
\frac{{((n/a)!)}^{a}a!}{2} \Big)}^{m} \leq (((n/2)-2)!)((n/2)!)
{\Big( \frac{(((n/2)-1)!)(((n/2)+1)!)}{2} \Big)}^{m-1}.$$

\item If $n$ is congruent to $2$ modulo $4$, then $$(1 + \alpha(m)
){\Big( \frac{{((n/a)!)}^{a}a!}{2} \Big)}^{m} \leq (1/2^{m-1})
{(((n/2)-1)!)}^{2} {((n/2)!)}^{2m-2}.$$
\end{enumerate}
\end{lem}

\begin{proof}
By Lemma \ref{l2} it is sufficient to show that both displayed
inequalities follow from the inequality
$$n {((n/a)!)}^{a}a! \leq 2 {((n/2)!)}^{2}.$$

Indeed, the first displayed inequality becomes
$$(1+\alpha(m)) {\Big(
\frac{{((n/a)!)}^{a}a!}{2} \Big)}^{m} \leq \frac{8}{n^{2}-4}
{\Big( \frac{(((n/2)-1)!)(((n/2)+1)!)}{2} \Big)}^{m}.$$ Since
${(1+\alpha(m))}^{1/m} \leq \sqrt{2}$ and $(2\sqrt{2})/n \leq
{(8/(n^{2}-4))}^{1/2} \leq {(8/(n^{2}-4))}^{1/m}$, it is
sufficient to see that
$$(n/2){((n/a)!)}^{a}a! \leq ((n/2)-1)!((n/2)+1)!.$$ But this
proves the first part of the lemma since $${((n/2)!)}^{2} <
((n/2)-1)!((n/2)+1)!.$$

After rearranging the factors in the second displayed inequality
of the statement of the lemma, we get
$$(1+\alpha(m)){\Big( {((n/a)!)}^{a}a! \Big)}^{m} \leq (8/n^{2}) {(n/2)!}^{2m}.$$
By similar considerations as in the previous paragraph, we see
that this latter inequality follows from the inequality $n
{((n/a)!)}^{a}a! \leq 2 {((n/2)!)}^{2}$.
\end{proof}

Now let $M$ be a maximal primitive subgroup of $A_n$. We know that
$|M| < 2.6^{n}$ by \cite{Ma1}. The following lemma is necessary
for our purposes.

\begin{lem}
\label{11.4} For $n > 12$ and $m \geq 2$ we have the following.
\begin{enumerate}
\item Let $n$ be odd with smallest prime divisor $p$ at most
$\sqrt[3]{n}$. Then
$$(1+ \alpha(m))2.6^{nm} \leq (1/(2^{m-1}n)) {\Big( {(n/p)!}^{p}p!
\Big)}^{m}.$$

\item If $n$ is divisible by $4$, then $$(1+ \alpha(m) )2.6^{nm}
\leq (((n/2)-2)!)((n/2)!) {\Big(
\frac{(((n/2)-1)!)(((n/2)+1)!)}{2} \Big)}^{m-1}.$$

\item If $n$ is congruent to $2$ modulo $4$, then
$$(1+ \alpha(m) )2.6^{nm} \leq (1/2^{m-1}){(((n/2)-1)!)}^{2}
{((n/2)!)}^{2m-2}.$$
\end{enumerate}
\end{lem}

\begin{proof}
By Lemma \ref{l1}, there is nothing to prove for $n \geq 17$ since
$$(1+ \alpha(m) )2.6^{nm} < (1+ \alpha(m) ) {(n!/2)}^{m/2}$$ holds
for $n \geq 17$. One can check the validity of the inequalities
for $n = 16$ and $n=14$ by hand.
\end{proof}

Putting together the results of the previous three sections, the
proof of Proposition \ref{p2} is complete by Lemma \ref{folklore}.

\section{A lower bound for $\sigma(A_{n} \wr C_{m})$}

In this section we show that if $n > 12$ and $n$ is not congruent
to $2$ modulo $4$, then
$$\alpha(m) + \frac{1}{2} \overset{n}{\underset{i \;
\mathrm{odd}}{\underset{i=1}{\sum}}} {{n \choose i}}^{m} <
\sigma(A_{n} \wr C_{m}).$$

To show this for $n$ divisible by $4$ and $n > 12$, notice that


$$\alpha(m) + \frac{1}{2} \overset{n}{\underset{i \;
\mathrm{odd}}{\underset{i=1}{\sum}}} {{n \choose i}}^{m} =
\sigma(\Pi) \leq \sigma(A_{n} \wr C_{m}).$$

Let $n > 12$ be odd. By \cite{Ma2} we may assume that $m > 1$. In
this case we clearly have
$$\alpha(m) + \frac{1}{2} \overset{n}{\underset{i \;
\mathrm{odd}}{\underset{i=1}{\sum}}} {{n \choose i}}^{m} <
2^{nm-m-1}.$$ Hence it is sufficient to show that $2^{nm-m-1} \leq
\sigma(A_{n} \wr C_{m})$.

We have $|\Pi_{1}| = (n-1)!{(n!/2)}^{m-1}$. Let $H = N_{G}(M
\times M^{g_{2}} \times \cdots \times M^{g_{m}})$ for some maximal
subgroup $M$ of $A_n$ and some elements $g_{2}, \ldots, g_{m} \in
A_n$. If $M$ is intransitive, then $\Pi_{1} \cap H = \emptyset$.
If $M$ is imprimitive, then, by Lemma \ref{11.1}, $$|\Pi_{1} \cap
H| \leq (1/(n2^{m-1})){(n/p)!}^{mp}{p!}^{m}$$ where $p$ is the
smallest prime divisor of $n$. If $M$ is primitive, then, by the
statement just before Lemma \ref{11.4}, $|\Pi_{1} \cap H| \leq
2.6^{nm}$. Now let $H$ be a subgroup of $G$ of diagonal type. Then
$|\Pi_{1} \cap H| \leq {(n!/2)}^{m/2}$. If $H$ is a maximal
subgroup of $G$ containing the socle of $G$, then $\Pi_{1} \cap H
= \emptyset$. Let $\mathcal{M}$ be a minimal cover (a cover with
least number of members) of $G$ containing maximal subgroups of
$G$. Let $a$ be the number of subgroups in $\mathcal{M}$ of the
form $N_{G}(M \times M^{g_{2}} \times \cdots \times M^{g_{m}})$
where $M$ is imprimitive. Let $b$ be the number of subgroups in
$\mathcal{M}$ of the form $N_{G}(M \times M^{g_{2}} \times \cdots
\times M^{g_{m}})$ where $M$ is primitive. Let $c$ be the number
of subgroups in $\mathcal{M}$ of diagonal type. Then $$a \cdot
(1/(n2^{m-1})){(n/p)!}^{mp}{p!}^{m} + b \cdot 2.6^{nm} + c \cdot
{(n!/2)}^{m/2} \geq (n-1)!{(n!/2)}^{m-1}.$$ From this we see that
$$\frac{(n-1)!{(n!/2)}^{m-1}}{\max \{
(1/(n2^{m-1})){(n/p)!}^{mp}{p!}^{m}, 2.6^{nm}, {(n!/2)}^{m/2} \}}
\leq \sigma(G)$$ if $n$ is not a prime, and
$$\frac{(n-1)!{(n!/2)}^{m-1}}{\max \{ 2.6^{nm},
{(n!/2)}^{m/2} \}} \leq \sigma(G)$$ if $n$ is a prime. Hence to
finish the proof of this section, it is sufficient to see

\begin{lem}
For $n > 12$ odd and for $m >1$ we have the following.
\begin{enumerate}
\item $$2^{nm-m-1} \leq \frac{(n!)^{m}}{{(n/p)!}^{mp}{p!}^{m}}$$
where $n$ is not a prime and $p$ is the smallest prime divisor of
$n$.

\item $$2^{nm-2} \leq \frac{(n-1)!{(n!)}^{m-1}}{2.6^{nm}}.$$

\item $$2^{nm-(m/2)-2} \leq (n-1)!{(n!)}^{(m/2)-1}.$$
\end{enumerate}
\end{lem}

\begin{proof}
(1) It is sufficient to prove the inequality $$2^{n-1} \leq
\frac{n!}{{(n/p)!}^{p}p!}$$ for $n \geq 15$. This is true by
inspection for $15 \leq n < 99$. Hence assume that $n \geq 99$.
Applying Stirling's formula (see Theorem \ref{Stirling}) three
times to both sides of the inequality $$2^{n-1}{(n/p)!}^{p}p! \leq
n!$$ we obtain $$2^{n-1}{\sqrt{2\pi (n/p)}}^{p} {(n/pe)}^{n}
e^{1/(12(n/p))} \sqrt{2 \pi p} {(p/e)}^{p} e^{1/12p} \leq
\sqrt{2\pi n}{(n/e)}^{n}e^{1/(12n+1)}.$$ Since
$e^{1/(12(n/p))}e^{1/12p} < 2$ and $e^{1/(12n+1)} > 1$, it is
sufficient to prove the inequality $$2^{n}{\sqrt{2\pi (n/p)}}^{p}
{(n/pe)}^{n}\sqrt{2 \pi p} {(p/e)}^{p} \leq \sqrt{2\pi
n}{(n/e)}^{n}.$$ After rearranging factors and applying the
estimate $3 \leq p \leq \sqrt{n}$ we see that it is sufficient to
prove $$2^{n} {\sqrt{2\pi n /3}}^{\sqrt{n}} \sqrt{2\pi \sqrt{n}}
{(\sqrt{n}/e)}^{\sqrt{n}} \leq 3^{n} \sqrt{2 \pi n}.$$ After
taking logarithms of both sides of the previous inequality and
rearranging terms, we get
$$(\sqrt{n}/2) \ln (2\pi n/3) + (1/2)\ln (2\pi \sqrt{n}) +
\sqrt{n} \ln (\sqrt{n}/e) \leq n \ln (3/2) + (1/2)\ln (2\pi n).$$
After further rearrangements we obtain
$$(\sqrt{n}-(1/4)) \ln n \leq \ln (3/2) n + \sqrt{n} (1- (\ln (2\pi/3)/2)).$$
After dividing both sides of the previous inequality by $\sqrt{n}$
and evaluating the logarithms we see that it is sufficient to
prove $\ln n \leq 0.4 \sqrt{n} + 0.63$ for $n \geq 99$. But this
is clearly true.

(2) Rearranging the inequality we get $(n/4)2^{nm} \leq
{(n!)}^{m}/2.6^{nm}$. Hence it is sufficient to see that
$(\sqrt{n}/2)5.2^{n} \leq n!$. But this is true for $n \geq 13$.

(3) Rearranging the inequality we get $(n/4)2^{nm-(m/2)} \leq
{(n!)}^{m/2}$. Hence it is sufficient to see that $(n/8)4^{n}\leq
n!$. But this is true for $n \geq 13$.
\end{proof}

\section{Proofs of Theorems \ref{main2} and \ref{main3}}

Let us first show Theorem \ref{main2}. Suppose that $n$ is
congruent to $2$ modulo $4$. If $n \geq 10$, then $\sigma(A_{n}) =
2^{n-2}$, by \cite{Ma2}. Hence we may assume that $m > 1$ (and $n
> 10$). In this case, by Proposition \ref{p2}, $\mathcal{H}$ is
definitely unbeatable on $\Pi$ and $\mathcal{H}_{0}$ is a covering
for $A_n$. Hence
$$\alpha(m) + \sum_{M \in \mathcal{H}_{0}} {|A_{n}:M|}^{m-1} =
|\mathcal{H}| = \sigma(\Pi) \leq \sigma(G) \leq \alpha(m) +
\sum_{M \in \mathcal{H}_{0}} {|A_{n}:M|}^{m-1},$$ by Proposition
\ref{bdsn}. Finally, it is easy to see that $$\alpha(m) + \sum_{M
\in \mathcal{H}_{0}} {|A_{n}:M|}^{m-1} = \alpha(m) +
\overset{(n/2)-2}{\underset{i \;
\mathrm{odd}}{\underset{i=1}{\sum}}} {{n \choose i}}^{m} +
\frac{1}{2^{m}} {{n \choose n/2}}^{m}.$$ This (and the previous
section) proves Theorem \ref{main2}.

>From now on assume that $n$ is either at least $16$ and divisible
by $4$ or odd with a prime divisor at most $\sqrt[3]{n}$. In this
case $\mathcal{H}$ is definitely unbeatable on $\Pi$ by
Proposition \ref{p2}. This gives us the lower bound
$$\alpha(m) + \sum_{M \in \mathcal{H}_{0}} {|A_{n}:M|}^{m-1} \leq \sigma(G).$$
Let the set $\mathcal{H}_{3}$ of maximal subgroups of $A_n$ be
defined as follows. If $4$ divides $n$, then let $\mathcal{H}_{3}$
be the set of all subgroups conjugate (in $A_n$) to $(S_{n/2} \wr
S_{2}) \cap A_n$. If $n$ is odd, then let $\mathcal{H}_{3}$ be the
set of all subgroups conjugate (in $A_n$) to some subgroup $(S_{k}
\times S_{n-k}) \cap A_n$ for some $k$ with $k \leq n/3$. Then
$\mathcal{H}_{0} \cup \mathcal{H}_{3}$ is a covering for $A_n$.
Hence, by Proposition \ref{bdsn}, this gives us the upper bound
$$\sigma(G) \leq \alpha(m) + \sum_{M \in \mathcal{H}_{0}} {|A_{n}:M|}^{m-1} +
\sum_{M \in \mathcal{H}_{3}} {|A_{n}:M|}^{m-1}.$$ Hence to prove
Theorem \ref{main3}, it is sufficient to see that the fraction
$$f(n,m) = \frac{ \sum_{M \in \mathcal{H}_{3}} {|A_{n}:M|}^{m-1} }
{ \sum_{M \in \mathcal{H}_{0}} {|A_{n}:M|}^{m-1}}$$ tends to $0$
as $n$ goes to infinity.

If $n$ is divisible by $4$, then $$f(n,m) = \frac{{((1/2){n
\choose n/2})}^{m}}{(1/2) \overset{n}{\underset{i
\mathrm{odd}}{\underset{i=1}{\sum}}} {n \choose i}^{m}}$$ which
clearly tends to $0$ as $n$ goes to infinity.

Finally, if $n$ is odd with smallest prime divisor $p$ at most
$\sqrt[3]{n}$, then
$$f(n,m) = \frac{\sum_{i=1}^{[n/3]}{n \choose i}^{m} }{{(n!/({(n/p)!}^{p}p!))}^{m}}
\leq \frac{\sum_{i=1}^{[n/3]}{n \choose i}^{m} }{ 2^{nm-m} } \leq
\frac{{\Big( \sum_{i=1}^{[n/3]}{n \choose i} \Big)}^{m} }{
2^{nm-m} }$$ which again tends to $0$ as $n$ goes to infinity.

This proves Theorem \ref{main3}.

\bigskip

\centerline{\bf Acknowledgements \rm}

\medskip

Thanks are due to Andrea Lucchini for helpful comments and to the
anonymous referee for a careful reading of an early version of
this paper.

\bigskip

{\it Martino Garonzi, Dipartimento di Matematica Pura ed
Applicata, Via Trieste

63, 35121 Padova, Italy.

E-mail address: mgaronzi@math.unipd.it}

\medskip

{\it Attila Mar\'oti, Alfr\'ed R\'enyi Institute of Mathematics,
Budapest, Hungary.

E-mail address: maroti@renyi.hu}

\end{document}